\documentclass[a4paper,11pt]{amsart}
\usepackage[utf8]{inputenc}

\usepackage{amsmath}  
\usepackage{amssymb}     
\usepackage{amsthm}  
\usepackage{amsfonts}
\usepackage{bbold}  
\usepackage{mathbbol}    
\usepackage{frenchineq}

\DeclareSymbolFontAlphabet{\mathbb}{AMSb}
\DeclareSymbolFontAlphabet{\mathbbl}{bbold}

\usepackage{bm}
\usepackage{hyperref}
\usepackage{accents}
\usepackage{mathrsfs}
\usepackage{mathtools}
\usepackage{fixmath}
\usepackage{fullpage}
\usepackage{paralist}
\usepackage{stmaryrd}
\usepackage{xspace}
\usepackage{bold-extra}

\usepackage{booktabs}
\usepackage{array}

\usepackage{pgf,tikz}
\usetikzlibrary{arrows}
\usetikzlibrary{patterns}
\usetikzlibrary[positioning]
\usetikzlibrary[fit]
\usetikzlibrary{shapes.geometric}
\usetikzlibrary{calc}
\usetikzlibrary{decorations.pathreplacing}

\newcommand{\eg}{e.g.}
\newcommand{\ie}{i.e.}			
\newcommand{\etal}{et al.}

\newcommand{\eps}{\varepsilon}
\newcommand{\T}{\mathbb{T}}

\newcommand{\R}{\mathbb{R}}
\newcommand{\transpose}[1]{{#1}^\top}
\newcommand{\Pcal}{\mathcal{P}}

\newcommand{\Qcal}{\mathcal{Q}}

\newcommand{\Ccal}{\mathcal{C}}

\newcommand{\Scal}{\mathcal{S}}

\newcommand{\Kcal}{\mathcal{K}}

\newcommand{\K}{\mathbb{K}}

\newcommand{\LP}{\mathbf{LP}}

\DeclareMathOperator{\val}{\mathrm{val}}

\DeclareMathOperator{\interior}{\mathrm{int}}

\DeclareMathOperator{\vol}{\mathrm{vol}}

\DeclareMathOperator{\rec}{\mathrm{rec}}

\newtheorem{theorem}{Theorem}
\newtheorem*{theorem*}{Theorem}
\newtheorem{proposition}[theorem]{Proposition}

\newtheorem{corollary}[theorem]{Corollary}
\newtheorem{lemma}[theorem]{Lemma}

\newtheorem{assumption}{Assumption}

\theoremstyle{remark}

\newcommand{\scalar}[2]{\langle#1,#2\rangle}

\newcommand{\tscalar}[2]{\langle #1, #2 \rangle_\T}

\newcounter{tikzbrace}
\newcommand{\tikzmark}[1]{\tikz[overlay,remember picture] {\node (brace-\thetikzbrace) {};}\stepcounter{tikzbrace}}

\newcommand{\insertbigbrace}[1]{%
\begin{tikzpicture}[remember picture, overlay]
\draw[thick] 
	let \n1 = {\thetikzbrace - 2},
		\n2 = {\thetikzbrace - 1},   
		\p1 = (brace-\n1),
		\p2 = (brace-\n2),
		\n3 = {max(\x1,\x2)},
		\p3 = ($(\n3,\y1) + (0.2,0.4)$),
		\p4 = ($(\n3,\y2) + (0.2,-0.2)$),
		\p5 = (-0.1,0) in 
	(\p3) ++ (\p5) -- (\p3) -- node[right=1ex]{#1} (\p4) -- ++ (\p5);
\end{tikzpicture}}

\usepackage[colorinlistoftodos,bordercolor=orange,backgroundcolor=orange!20,linecolor=orange,textsize=scriptsize]{todonotes}
\newlength{\mytemplen}

\title{The tropicalization of the entropic barrier}

\author{Xavier Allamigeon
\and Abdellah Aznag
\and Stéphane Gaubert
\and Yassine Hamdi} 

\date{\today}

\address[Xavier Allamigeon \and Stéphane Gaubert]{INRIA and CMAP, CNRS, \'Ecole Polytechnique, Institut Polytechnique de Paris, 91120 Palaiseau, France}
\email{firstname.lastname@inria.fr}

\address[Abdellah Aznag]{Industrial Engineering and Operations Research, Columbia University, New York, NY}
\email{aa4693@columbia.edu}

\address[Yassine Hamdi]{\'Ecole Polytechnique, Institut Polytechnique de Paris, 91120 Palaiseau, France} \email{yassine.hamdi@polytechnique.edu}

\thanks{This work was partially done when the second author was with Ecole Polytechnique, Institut Polytechnique de Paris.}

\allowdisplaybreaks 

\begin{document}

\begin{abstract}
The entropic barrier, studied by Bubeck and Eldan
(Proc.\ Mach.\ Learn.\ Research, 2015), is a self-concordant barrier with asymptotically optimal self-concordance parameter. In this paper, we study the tropicalization of the central path associated with the entropic barrier, \ie, the logarithmic limit of this central path for a parametric family of linear programs defined over the field of Puiseux series. Our main result is that the tropicalization of the entropic central path is a piecewise linear curve which coincides with the tropicalization of the logarithmic central path studied by Allamigeon et al.\ (SIAM J.\ Applied Alg.\ Geom., 2018). One consequence is that the number of linear pieces in the tropical entropic central path can be exponential in the dimension and the number of inequalities defining the linear program.
\end{abstract}

\maketitle

\section{Introduction}

The \emph{entropic barrier} on a convex body $\Kcal \subset \R^n$ is defined as the Cramér transform of the characteristic function of $\Kcal$. This means that it is the Fenchel conjugate $f^*_\Kcal$ of the logarithmic Laplace transform $f_\Kcal$ of the characteristic function of $\Kcal$, \ie,
\[
f_\Kcal(\theta) \coloneqq \log \int_\Kcal e^{\scalar{\theta}{x}} dx \, .
\]
The interest for the entropic barrier comes from the following theorem established by Bubeck and Eldan, which states that $f^*_\Kcal$ is a self-concordant barrier whose self-concordance parameter is asymptotically optimal:
\begin{theorem*}[\cite{BubeckEldan15}]
Let $\Kcal \subset \R^n$ be a convex body. The entropic barrier is a $(1+\eps_n) n$-self-concordant barrier on $\Kcal$, with $\eps_n \leq 100 \sqrt{\log(n)/n}$, for any $n \geq 80$. 
\end{theorem*}

Self-concordant barriers are the cornerstone of path-following interior point methods. In more detail, consider a linear program of the form
\begin{equation}\label{eq:LP}
\text{minimize} \enspace \scalar{c}{x} \enspace \text{subject to} \enspace x \in \Pcal \, .
\end{equation}
where $\Pcal \subset \R^n$ is a convex polyhedron and $c \in \R^n$. Every self-concordant barrier $\phi$ over $\Pcal$ gives rise to a function called the \emph{central path}, which maps any positive real number $\mu$ to the (unique) optimal solution of the convex program
\begin{equation}\label{eq:central_path}
\text{minimize} \enspace \scalar{c}{x} + \mu \phi(x) \enspace \text{subject to} \enspace x \in \interior \Pcal \, .
\end{equation}
The central path converges towards an optimal solution of~\eqref{eq:LP} when $\mu \to 0^+$, and the basic principle of interior point methods is to follow the central path in an approximate way for decreasing values of the parameter $\mu$ down to $0$. More information can be found in
the monography of Renegar~\cite{Renegar01},
which includes a complete introduction to the topic of path-following interior point methods, and in the book
of Nesterov and Nemirovskii~\cite{NesterovN94}, in which the theory of self-concordant barriers is developed.

The question
of the complexity of path-following interior point methods has received much attention over the last thirty years. One of the main questions, motivated by Smale's ninth problem~\cite{Smale98}, is whether interior point methods can solve linear programming in strongly polynomial complexity. The number of iterations performed by path-following interior point methods is intimately related with the properties of the barrier function and the resulting central path. For example, the number of iterations to get an $\eps$-approximation of the optimal value is bounded by $O(\sqrt \vartheta \log (1 / \eps))$, where $\vartheta$ is the so-called self-concordance parameter of the barrier; see~\cite[Chapter~2.3]{Renegar01} for a definition. This has motivated many works on the improvement of this parameter, \eg, \cite{LeeSidford14,BubeckEldan15,LeeYue18,LeeSidford19}. In a recent work~\cite{ABGJ18}, Allamigeon~\etal{}\ have shown that interior point methods based on the logarithmic barrier are not strongly polynomial, by studying the ``tropicalization'' of the associated  central path. The tropical central path is defined as a log-limit of the central path of a parametric family of linear programs; it follows from results of tropical geometry that it is a piecewise-linear curve. It is shown in~\cite{ABGJ18} that the number of ``linear pieces'' of the tropical central path can be exponential in the dimension and the number of inequalities defining these linear programs, and that this entails that log-barrier interior point methods can make a number of iterations that is exponential in the same quantities. It remained an open question whether other barriers could be studied in this framework. 

In this short note, we solve this problem in the case of the entropic barrier, by establishing a characterization of the tropicalization of the associated central path. We denote by $\Ccal_{\Pcal, c}(\mu)$ the point of the entropic central path, \ie~the unique optimal solution of~\eqref{eq:central_path} when $\phi = f^*_\Pcal$. 

As in~\cite{ABGJ18}, our approach makes use of a real-closed and nonarchimedean field of convergent series, here the field $\K$ of \emph{absolutely convergent generalized real Puiseux series}. The latter consists of series of the form $\sum_{\alpha \in \R} c_\alpha t^\alpha$ such that:
\begin{inparaenum}[(i)]
\item either the support $\{ \alpha \in \R \colon c_\alpha \neq 0\}$ is finite or $-\infty$ is its unique accumulation point;
\item the coefficients $c_\alpha$ are real, and the series is absolutely convergent for any sufficiently large values of $t \in \R$ (which we denote $t \gg 1$).
\end{inparaenum}
The \emph{valuation} of an element $\bm z \in \K$, denoted by $\val \bm z$, is defined as the greatest element $\alpha$ in the support of $\bm z$ (or $-\infty$ if $\bm z$ is the zero series). Equivalently, the valuation of $\bm z$ corresponds to the limit when $t \to +\infty$ of $\log_t |\bm z(t)|$, where $\log_t (\cdot) \coloneqq \frac{1}{\log t}\log(\cdot)$ is the logarithm map in base $t$. 

In this setting, any polyhedron $\bm \Pcal = \{ \bm x \in \K^n \colon \bm A \bm x \leq \bm b \}$ (where $\bm A \in \K^{m \times n}$ and $\bm b \in \K^m$) and objective vector $\bm c \in \K^n$ defined by entries over the field of Puiseux series give rise to a parametric family $\LP(t)$ of linear programs
\begin{equation}
\text{minimize} \enspace \scalar{\bm c(t)}{x} \enspace \text{subject to} \enspace x \in \bm \Pcal(t) \, , \label{eq:param_LP}
\end{equation}
where $\bm \Pcal(t)$ is defined as the polyhedron $\{ x \in \R^n \colon \bm A (t) x \leq \bm b(t)\}$ for all $t \gg 1$. The main result of the paper is the following characterization of the tropicalization (\ie, the log-limit) of their entropic central paths~$\Ccal_{\bm \Pcal(t), \bm c(t)}$:
\begin{theorem}\label{th:main}
Suppose that Assumption~\ref{assump} holds (see Section~\ref{sec:main}). For all $\bm \mu \in \K_{> 0}$, the following limit
\begin{align}
\lim_{t \to +\infty} \log_t \Ccal_{\bm \Pcal(t), \bm c(t)}(\bm \mu (t)) \label{eq:limit} && \text{(where $\log_t(\cdot)$ is understood entrywise)}
\end{align}
exists and is equal to the tropical barycenter of the tropical polytope $\{ x \in \val \bm \Pcal \colon \tscalar{\val \bm c}{x} \leq \val \bm \mu \}$, where $\tscalar{x}{y} \coloneqq \max_i (x_i + y_i)$. 
\end{theorem}
The definition of tropical polytopes and their barycenters, as well as background on polyhedra over Puiseux series, are given in Section~\ref{sec:prelim}. The rest of the paper is organized as follows. Theorem~\ref{th:main} is proved in Section~\ref{sec:main}. In Section~\ref{sec:cor}, we state that the tropicalization of the entropic central path coincides with that of the logarithmic central path studied in~\cite{ABGJ18}. Then, we adapt the pathological instance of linear programs built in~\cite{ABGJ18} to satisfy Assumption~\ref{assump} and show that their tropical entropic central path has exponentially many linear pieces.

\section{Preliminaries on polyhedra over Puiseux series and their tropicalization}\label{subsec:puiseux} \label{sec:prelim}

We recall that the field $\K$ naturally comes with a total order $\leq$, defined by $\bm x \leq \bm y$ if $\bm x = \bm y$ or if the  coefficient of the leading term (\ie, the term with greatest exponent) in the series $\bm y - \bm x$ is positive. Equivalently, $\bm x \leq \bm y$ if and only if $\bm x(t) \leq \bm y(t)$ for all $t$ large enough. The real-closed character of the field $\K$ is established in~\cite{Dries1998}.

As mentioned in the introduction, we call a \emph{polyhedron over Puiseux series} a set of the form $\bm \Pcal = \{ \bm x \in \K^n \colon \bm A \bm x \leq \bm b \}$ where $\bm A \in \K^{m \times n}$ and $\bm b \in \K^m$ for some $m \in \mathbb{N}$. By the previous discussion, remark that if $\bm x \in \bm \Pcal$, then $\bm x(t) \in \bm \Pcal(t)$ for $t \gg 1$. The following proposition states an analogous property on the inclusion of polyhedra over Puiseux series:
\begin{proposition}\label{prop:inclusion}
Let $\bm \Pcal, \bm \Qcal \subset \K^n$ two polyhedra over Puiseux series such that $\bm \Pcal \subset \bm \Qcal$. Then $\bm \Pcal(t) \subset \bm \Qcal(t)$ for all $t \gg 1$.
\end{proposition}

\begin{proof}
It suffices to show that the result holds when $\bm \Qcal$ is a halfspace, \ie, $\bm \Qcal = \{ \bm x \in \K^n \colon \scalar{\bm c}{\bm x} \leq \bm \alpha \}$ where $\bm c \in \K^n$ and $\bm \alpha \in \K$. Farkas Lemma still holds over the field of $\K$. Indeed, the statement of Farkas Lemma can be written as a first-order sentence in the language of ordered fields. Since this sentence is valid over $\R$, it is valid over any model of the complete theory of real-closed field, including $\K$. Therefore, supposing that $\bm \Pcal = \{ \bm x \in \K^n \colon \bm A \bm x \leq \bm b \}$ with $\bm A \in \K^{m \times n}$ and $\bm b \in \K^m$, there exists $\bm \lambda \in \K_{\geq 0}^m$ such that:
\begin{enumerate}[(i)]
\item $\transpose{\bm A} \bm \lambda = 0$ and $\scalar{\bm b}{\bm \lambda} < 0$;
\item or $\transpose{\bm A} \bm \lambda = \bm c$ and $\scalar{\bm b}{\bm \lambda} \leq \bm \alpha$.
\end{enumerate}
Consequently, the previous property holds by replacing $\bm A$, $\bm b$, $\bm c$, $\bm \lambda$ and $\bm \alpha$ by $\bm A(t)$, $\bm b(t)$, $\bm c(t)$, $\bm \lambda(t)$ and $\bm \alpha(t)$ respectively, as soon as $t \gg 1$. We deduce that $\bm \Pcal(t) \subset \bm \Qcal(t)$ for such $t$.
\end{proof}

The valuation map constitutes a monotone homomorphism from the semifield $\K_{\geq 0}$ of nonnegative Puiseux series to the \emph{tropical (max-plus) semifield} $\T \coloneqq (\R \cup \{-\infty\}, \vee, +)$, where we denote $x \vee y \coloneqq \max(x,y)$. In consequence, the valuation of the scalar product $\scalar{\bm x}{\bm y} = \sum_{i = 1}^n \bm x_i \bm y_i$ of two vectors $\bm x, \bm y \in \K_{\geq 0}^n$ is given by $\tscalar{\val \bm x}{\val \bm y} = \bigvee_{i = 1}^n (\val \bm x_i + \val \bm y_i)$, where we extend the valuation to elements of $\K^n$ entrywise.

A \emph{tropical polyhedron} is a set of points $x \in \T^n$ satisfying finitely many inequalities of the form $\alpha_0 \vee (\alpha_1 + x_1) \vee \dots \vee (\alpha_n + x_n) \leq  \beta_0 \vee (\beta_1 + x_1) \vee \dots \vee (\beta_n + x_n) $, where $\alpha_i, \beta_i \in \T$. (Note that the latter are the analogues of linear (affine) inequalities in the tropical semifield.) As shown in~\cite[Proposition~7]{ABGJ18} and~\cite[Proposition~2.6]{ABGJ15}, tropical polyhedra are precisely the images under the valuation map of polyhedra over Puiseux series which are included in the nonnegative orthant $\K_{\geq 0}^n$. One of the basic properties of tropical polyhedra is that they are tropically convex. A set $\Scal \subset \T^n$ is said to be \emph{tropically convex} if for all $x, y \in \Scal$ and $\lambda, \mu \in \T$ such that $\lambda \vee \mu = 0$, the entrywise supremum $(\lambda \mathbbl{1} + x) \vee (\mu \mathbbl{1} + y)$ belongs to $\Scal$ (where we denote by $\mathbbl{1}$ the all-$1$ $n$-vector). As a consequence, every compact tropical convex set in $\T^n$ admits a greatest element (for the entrywise ordering over $\T^n$), which is called the \emph{tropical barycenter}. A \emph{tropical polytope} is a compact tropical polyhedron.

\begin{proposition}\label{prop:lift_interior}
Let $\bm \Pcal \subset \K_{\geq 0}^n$ a polyhedron over Puiseux series. If $x$ lies in the interior of $\val \bm \Pcal$, then every $\bm x \in \K_{\geq 0}^n$ satisfying $\val \bm x = x$ belongs to $\bm \Pcal$. 
\end{proposition}

\begin{proof}
We denote by $\mathbbl{e}^i$ the $i$th element of the canonical basis of $\R^n$. Take $\eps > 0$ such that the points $x - \eps \mathbbl{1}$ and $x + \eps \mathbbl{e}^i$ for $i \in [n]$ belong to $\val \bm \Pcal$, and let $\bm x^0, \bm x^1, \dots, \bm x^n$ be some elements of $\bm \Pcal$ mapped by the valuation to these points respectively. We claim that any point $\bm x \in \K_{\geq 0}^n$ satisfying $\val \bm x = x$ belongs to the simplex generated by the points $\bm x^0, \dots, \bm x^n$, which will complete the proof. The latter simplex is defined by the $n+1$ inequalities of the form
\[
\det \begin{pmatrix}
1 & \cdots & 1 \\
\bm x^0 & \cdots & \bm x^n 	 
\end{pmatrix}
\det \begin{pmatrix}
1 & \cdots & 1 & 1 & 1 & \cdots & 1 \\
\bm x^0 & \dots & \bm x^{i-1} & \bm x & \bm x^{i+1} & \cdots & \bm x^n 
\end{pmatrix} \geq 0 \quad (i = 0, \dots, n)
\]
or, equivalently,
\begin{equation}\label{eq:ineq}
\det \begin{pmatrix}
t^\eps & 1 & \cdots & 1 \\
t^\eps \bm x^0 & \bm x^1 & \cdots & \bm x^n 	 
\end{pmatrix}
\det \begin{pmatrix}
t^\eps & 1 & \cdots & 1 & 1 & 1 & \cdots & 1 \\
t^\eps \bm x^0 & \bm x^1 & \dots & \bm x^{i-1} & \bm x & \bm x^{i+1} & \cdots & \bm x^n 
\end{pmatrix} \geq 0 \quad (i = 0, \dots, n) \, .
\end{equation}
Developing the determinant $\det \begin{pmatrix}
t^\eps & 1 & \cdots & 1 \\
t^\eps \bm x^0 & \bm x^1 & \cdots & \bm x^n
\end{pmatrix}$ (using the Leibniz formula) shows that the term $t^\eps \bm x^1_1 \cdots \bm x^n_n$ is the unique one with maximal valuation (equal to $x_1 + \dots + x_n + (n+1) \eps$), and thus this determinant is positive. We now consider $\bm x \in \K_{\geq 0}^n$ such that $\val \bm x = x$, and $i \in \{0, \dots, n\}$. In the determinant $\det \begin{pmatrix}
t^\eps & 1 & \cdots & 1 & 1 & 1 & \cdots & 1 \\
t^\eps\bm x^0 & \bm x^1 & \dots & \bm x^{i-1} & \bm x & \bm x^{i+1} & \cdots & \bm x^n 
\end{pmatrix}$, the term $t^\eps \bm x^1_1 \cdots \bm x^{i-1}_{i-1} \bm x_i \bm x^{i+1}_{i+1} \cdots \bm x^n_n$ has valuation $x_1 + \dots + x_n + n\eps$, while the valuation of the other terms is less than or equal to $x_1 + \dots + x_n + (n-1) \eps$. We deduce that this determinant is also positive, thus all the inequalities~\eqref{eq:ineq} are satisfied.
\end{proof}

\section{Proof of Theorem~\ref{th:main}}\label{sec:main}

We first review some useful properties related with the entropic barrier. When $\Pcal$ is a nonempty but possibly unbounded polyhedron, the domain of the function $f_{\Pcal}$ is the interior of $(\rec \Pcal)^\circ$, where $\rec \Pcal \coloneqq \{ z \in \R^n \colon \exists x , \;  x + \R_{\geq 0} z \subset \Pcal \}$ is the recession cone of $\Pcal$, and $(\rec \Pcal)^\circ$ is its polar cone, \ie, $(\rec \Pcal)^\circ = \{ y \in \R^n \colon \forall z \in \rec \Pcal , \; \scalar{y}{z} \leq 0 \}$. Note that $\theta \in \interior\bigl((\rec \Pcal)^\circ\bigr)$ if and only if the function $x \mapsto \scalar{\theta}{x}$ is bounded from above over $\Pcal$, and the set of points of $\Pcal$ maximizing this function is compact.

The following proposition shows that we can easily express the point of the central path in terms of the gradient of the function $f_\Pcal$:
\begin{proposition}\label{prop:entropic_cp}
Suppose that $\Pcal$ has nonempty interior and $-c \in \interior\bigl((\rec \Pcal)^\circ\bigr)$. For all $\mu > 0$, we have $\Ccal_{\Pcal, c}(\mu) = \nabla f_{\Pcal}(-\mu^{-1} c)$, \ie,
\[
\Ccal_{\Pcal,c}(\mu) = \Bigl(\int_{\Pcal} e^{- \scalar{\mu^{-1} c}{x}} dx\Bigr)^{-1} \int_{\Pcal} e^{- \scalar{\mu^{-1} c}{x}} x \, dx \, .
\]
\end{proposition}

\begin{proof}
As a self-concordant barrier, the function $f^*_\Pcal$ is differentiable and strictly convex. Its domain is $\interior \Pcal$, and $f^*_\Pcal(x) \to \infty$ when $x$ tends to the boundary of $\Pcal$. Besides, the function $x \mapsto \scalar{c}{x} + \mu f^*_\Pcal(x)$ is bounded from below over $\interior \Pcal$ by $-\mu f_\Pcal(-\mu^{-1} c)$ (Fenchel--Young inequality), which is finite since $-\mu^{-1} c \in \interior\bigl((\rec \Pcal)^\circ\bigr)$. We deduce that the optimization problem~\eqref{eq:central_path} has a unique optimal point $x^* \in \interior \Pcal$, given by $\nabla f^*_\Pcal(x^*) = -\mu^{-1} c$. Note that, as a proper closed convex function, $f_\Pcal$ is equal to the Fenchel conjugate of $f^*_\Pcal$. Moreover, since the domain of $f^*_\Pcal$ has nonempty interior and the norm of $\nabla f^*_\Pcal(x)$ diverges to $+\infty$ when the sequence $x$ tends to the boundary of $\Pcal$, the function $f^*_\Pcal$ is Legendre. By~\cite[Theorem~26.5]{Rockafellar70}, the function $x \mapsto \nabla f^*_\Pcal(x)$ is invertible, with inverse $\theta \mapsto \nabla f_\Pcal(\theta)$. We deduce that $x^* = \nabla f_\Pcal(-\mu^{-1} c)$.
\end{proof}

We now suppose that the polyhedron $\bm \Pcal \subset \K^n$ and the objective vector $\bm c \in \K^n$ satisfy the following requirements:
\begin{assumption}\label{assump}
\begin{asparaenum}[(i)]
\item\label{assump:i} $\bm \Pcal \subset \K_{\geq 0}^n$ and $0 \in \bm \Pcal$;
\item\label{assump:ii} the image of $\bm \Pcal$ under the valuation map is a regular set, \ie, it is equal to the closure of its interior;
\item\label{assump:iii} $\bm c \in (\K_{> 0})^n$.
\end{asparaenum}
\end{assumption}
Observe that by Proposition~\ref{prop:inclusion}, we have $\bm \Pcal(t) \subset \R_{\geq 0}^n$ provided that $t$ is large enough, as well as $0 \in \bm \Pcal(t)$ and $\bm c_i > 0$ for all $i \in [n]$. For such $t$, we have $-\bm c(t) \in \interior\bigl((\rec \bm \Pcal(t))^\circ\bigr)$, so that Proposition~\ref{prop:entropic_cp} applies. Note that $-\infty \in \val \bm \Pcal$, and $\val \bm c_i > -\infty$ for all $i \in [n]$. The latter property implies that $\{ x \in \val \bm \Pcal \colon \tscalar{\val \bm c}{x} \leq \alpha \}$ is a compact tropical polyhedron, and thus its tropical barycenter, which we denote by $x^*(\alpha)$, is well-defined.

From now on, we fix $\bm \mu \in \K_{> 0}$ and denote $\mu \coloneqq \val \bm \mu$. Theorem~\ref{th:main} is a straightforward consequence of Proposition~\ref{prop:entropic_cp} and the next two lemmas.
\begin{lemma}\label{lemma:volume}
\[
\lim_{t \to \infty} \log_t \int_{\bm \Pcal(t)} e^{-\scalar{\bm \mu(t)^{-1} \bm c(t)}{x}} d x = \sum_{i = 1}^n x^*_i(\mu) \, .
\]
\end{lemma}

\begin{proof}
We introduce $\delta, \eps > 0$. We have:
\begin{equation}\label{eq:sum}
\int_{\bm \Pcal(t)} e^{-\scalar{\bm \mu(t)^{-1} \bm c(t)}{x}} d x  = 
\int_{\begin{subarray}{l}
\, x \in \bm \Pcal(t)\\
\scalar{\bm c(t)}{x} > \bm \mu(t) t^\delta
\end{subarray}}
e^{-\scalar{\bm \mu(t)^{-1} \bm c(t)}{x}} d x 
+
\int_{\begin{subarray}{l}
\, x \in \bm \Pcal(t)\\
\scalar{\bm c(t)}{x} \leq \bm \mu(t) t^\delta
\end{subarray}}
e^{-\scalar{\bm \mu(t)^{-1} \bm c(t)}{x}} d x \, .
\end{equation}
We first focus on the first term of the sum in~\eqref{eq:sum}. Since $\bm c_i(t) > 0$ for all $i \in [n]$ provided that $t \gg 1$, we can change the variable $x_n$ by $u = \scalar{\bm \mu(t)^{-1} \bm c(t)}{x}$, so that:
\begin{multline*}
\int_{\begin{subarray}{l}
\, x \in \bm \Pcal(t)\\
\scalar{\bm c(t)}{x} > \bm \mu(t) t^\delta
\end{subarray}}
e^{-\scalar{\bm \mu(t)^{-1} \bm c(t)}{x}} d x \\ 
\begin{aligned}[t]
& \leq
\frac{\bm \mu(t)}{\bm c_n(t)}\int_{t^\delta}^{\infty} 
\bigl(\vol\{(x_1, \dots, x_{n-1}) \in \R_{\geq 0}^{n-1} \colon 
\sum_{i = 1}^{n-1} \bm c_i(t) x_i \leq \bm \mu(t) u\}\bigr) e^{-u} \, du \\
& \leq \frac{\bm \mu(t)^n}{(n-1)! \, \bm c_1(t) \cdots \bm c_n(t)} \int_{t^\delta}^\infty u^{n-1} e^{-u} \, du 
\\
& 
\leq \frac{\bm \mu(t)^n}{(n-1)! \, \bm c_1(t) \cdots \bm c_n(t)} \Gamma(n,t^\delta) \, .
\end{aligned}
\end{multline*}
where $\Gamma$ is the upper incomplete Gamma function, defined by $\Gamma(k,z) \coloneqq \int_z^\infty u^{k-1} e^{-u} \, du$. Recall that $\Gamma(k,z) = z^{k-1} e^{-z} + (k-1) \Gamma(k-1,z)$ for all $k$, thus $\Gamma(n,z) = \text{poly}(z) e^{-z}$ where $\text{poly}(z)$ stands for some polynomial in $z$ (of degree $n-1$). We deduce that
\begin{equation}\label{eq:first_term}
\int_{\begin{subarray}{l}
\, x \in \bm \Pcal(t)\\
\scalar{\bm c(t)}{x} > \bm \mu(t) t^\delta
\end{subarray}}
e^{-\scalar{\bm \mu(t)^{-1} \bm c(t)}{x}} d x = O\Bigl(\frac{\bm \mu(t)^n}{\, \bm c_1(t) \cdots \bm c_n(t)} \text{poly}(t^\delta) e^{-t^\delta} \Bigr)\; .
\end{equation}

For the second term of the sum in~\eqref{eq:sum}, we remark that the image under the valuation map of the Puiseux polyhedron $\{ \bm x \in \bm \Pcal \colon \scalar{\bm c}{\bm x} \leq \bm \mu t^\delta \}$ is included in the tropical polytope $\{x \in \val \bm \Pcal \colon \tscalar{\val \bm c}{x} \leq \mu + \delta \}$. By definition of the tropical barycenter, for all $\bm x$ in the former polyhedron, we have $\val \bm x \leq x^*(\mu + \delta)$, hence $\bm x_i < t^{x^*_i(\mu + \delta) + \eps}$ for all $i \in [n]$. Then, Proposition~\ref{prop:inclusion} implies that, provided $t$ is large enough,
\begin{equation}\label{eq:inclusion}
\{ x \in \bm \Pcal(t) \colon \scalar{\bm c(t)}{x} \leq \bm \mu(t) t^\delta \} \subset \prod_{i = 1}^n [0, t^{x^*_i(\mu + \delta) + \eps}]
\end{equation}
Consequently, for $t \gg 1$, the second term of the sum in~\eqref{eq:sum} satisfies
\begin{equation}\label{eq:second_term}
\int_{\begin{subarray}{l}
\, x \in \bm \Pcal(t)\\
\scalar{\bm c(t)}{x} \leq \bm \mu(t) t^\delta
\end{subarray}} 
e^{-\scalar{\bm \mu(t)^{-1} \bm c(t)}{x}} d x 
\leq \vol 
\{ x \in \bm \Pcal(t) \colon \scalar{\bm c(t)}{x} \leq \bm \mu(t) t^\delta \} \leq 
 t^{\sum_{i = 1}^n x^*_i(\mu + \delta) + n \eps} \, .
\end{equation}
As $e^{-t^\delta} \mathrel{=}_{t \to +\infty} o(t^K)$ for all $K \in \R$, the asymptotics given in~\eqref{eq:first_term} of the first term of the sum in~\eqref{eq:sum} is negligible w.r.t.~that of the the second term. Therefore, we have
\begin{equation}\label{eq:upper}
\limsup_{t \to \infty} \; \log_t 
\int_{\bm \Pcal(t)} e^{-\scalar{\bm \mu(t)^{-1} \bm c(t)}{x}} d x \leq 
\sum_{i = 1}^n x^*_i(\mu + \delta) + n \eps \, .
\end{equation}

We now establish a lower bound. Since $\val \bm \Pcal$ is a tropical convex set that contains $-\infty$, the point $x^*(\mu) - \eps \mathbbl{1}$ belongs to $\val \bm \Pcal$. As $\val \bm \Pcal$ is regular, we can find $\tilde x \in \interior \val \bm \Pcal$ such that $\|\tilde x - (x^*(\mu) - \eps \mathbbl{1})\|_\infty \leq \eps$. Thus, $x^*(\mu) - 2\eps \mathbbl{1} \leq \tilde x \leq x^*(\mu)$. By Proposition~\ref{prop:lift_interior}, the parallelotope $t^{\tilde x} \leq \bm x \leq 2t^{\tilde x}$ is included in $\bm \Pcal$. Therefore, by Proposition~\ref{prop:inclusion}, for all $t \gg 1$, we have
\begin{equation}\label{eq:inclusion2}
\{ x \in \R^n \colon t^{\tilde x} \leq x \leq 2t^{\tilde x} \} \subset \bm \Pcal(t) \, .
\end{equation}
In consequence, 
\[
\log_t \int_{\bm \Pcal(t)} e^{-\scalar{\bm \mu(t)^{-1} \bm c(t)}{x}} d x
\geq \log_t \int_{\begin{subarray}{l}
\, x \in \R^n \\
t^{\tilde x} \leq x \leq 2t^{\tilde x}
\end{subarray}} e^{-\scalar{\bm \mu(t)^{-1} \bm c(t)}{x}} d x 
\geq - 2 \frac{\scalar{\bm \mu(t)^{-1} \bm c(t)}{t^{\tilde x}}}{\log t} + \sum_{i = 1}^n \tilde x_i
\]
Note that $\val \scalar{\bm \mu^{-1} \bm c}{t^{\tilde x}} = \tscalar{\val \bm c}{\tilde x} - \mu \leq \tscalar{\val \bm c}{x^*} - \mu \leq 0$. Therefore, $\scalar{\bm \mu(t)^{-1} \bm c(t)}{t^{\tilde x}} \mathrel{=}_{t \to \infty} O(1)$. We deduce that 
\begin{equation}\label{eq:lower}
\liminf_{t \to \infty} \; \log_t \int_{\bm \Pcal(t)} e^{-\scalar{\bm \mu(t)^{-1} \bm c(t)}{x}} d x \geq \sum_{i = 1}^n x^*_i(\mu) - 2 n \eps \, .
\end{equation}
The map $\alpha \mapsto x^*(\alpha)$ is continuous by~\cite[Corollary~17]{ABGJ18}. Passing to the limit $\delta, \eps \to 0^+$ in~\eqref{eq:upper} and~\eqref{eq:lower} shows that the limit of $\log_t \int_{\bm \Pcal(t)} e^{-\scalar{\bm \mu(t)^{-1} \bm c(t)}{x}} d x$ exists and is equal to $\sum_{i = 1}^n x^*_i(\mu)$.
\end{proof}

\begin{lemma}\label{lemma:mean}
\[
\lim_{t \to \infty} \log_t \int_{\bm \Pcal(t)} e^{-\scalar{\bm \mu(t)^{-1} \bm c(t)}{x}} x \, d x = x^* (\mu) + \Bigl(\sum_{i = 1}^n x^*_i (\mu)\Bigr) \mathbbl{1} \, .
\]
\end{lemma}

\begin{proof}
The approach is very akin to the proof of Lemma~\ref{lemma:volume}. We fix $j \in [n]$ and introduce $\delta, \eps > 0$. We have:
\[
\int_{\bm \Pcal(t)} e^{-\scalar{\bm \mu(t)^{-1} \bm c(t)}{x}} x_j \, d x  = \\
\int_{\begin{subarray}{l}
\, x_i \in \bm \Pcal(t)\\
\scalar{\bm c(t)}{x} > \bm \mu(t) t^\delta
\end{subarray}}
e^{-\scalar{\bm \mu(t)^{-1} \bm c(t)}{x}} x_j \, d x 
+
\int _{\begin{subarray}{l}
\, x \in \bm \Pcal(t)\\
\scalar{\bm c(t)}{x} \leq \bm \mu(t) t^\delta
\end{subarray}}
e^{-\scalar{\bm \mu(t)^{-1} \bm c(t)}{x}} x_j \, d x \, . 
\]
Using the same change of the variable as before, we have:
\[
\int_{\begin{subarray}{l}
\, x \in \bm \Pcal(t)\\
\scalar{\bm c(t)}{x} > \bm \mu(t) t^\delta
\end{subarray}}
e^{-\scalar{\bm \mu(t)^{-1} \bm c(t)}{x}} x_j \, d x 
\leq \frac{\bm \mu(t)^{n+1}}{n! \, \bm c_j(t) \, \bm c_1(t) \cdots \bm c_n(t)} \Gamma(n+1,t^\delta) \, .
\]
Besides, the inclusion~\eqref{eq:inclusion} for $t \gg 1$ provides
\[
\int_{\begin{subarray}{l}
\, x \in \bm \Pcal(t)\\
\scalar{\bm c(t)}{x} \leq \bm \mu(t) t^\delta
\end{subarray}} 
e^{-\scalar{\bm \mu(t)^{-1} \bm c(t)}{x}} x_j \, d x 
\leq t^{x^*_j(\mu + \delta) + \eps + \sum_{i = 1}^n x^*_i(\mu + \delta) + n \eps} \, .
\]
We deduce that 
\begin{equation}\label{eq:upper2}
\limsup_{t \to \infty} \; \log_t 
\int_{\bm \Pcal(t)} e^{-\scalar{\bm \mu(t)^{-1} \bm c(t)}{x}} x_j \, d x \leq x^*_j(\mu + \delta) + \sum_{i = 1}^n x^*_i(\mu + \delta) + (n+1) \eps \, .
\end{equation}

For the converse inequality, we introduce $\tilde x$ as in the proof of Lemma~\ref{lemma:volume}, so that the inclusion~\eqref{eq:inclusion2} holds for $t \gg 1$. We have
\[
\log_t \int_{\bm \Pcal(t)} e^{-\scalar{\bm \mu(t)^{-1} \bm c(t)}{x}} x_j \, d x \geq \tilde x_j - 2 \frac{\scalar{\bm \mu(t)^{-1} \bm c(t)}{t^{\tilde x}}}{\log t} + \sum_{i = 1}^n \tilde x_i\, ,
\]
and so
\begin{equation}\label{eq:lower2}
\liminf_{t \to \infty} \; \log_t \int_{\bm \Pcal(t)} e^{-\scalar{\bm \mu(t)^{-1} \bm c(t)}{x}} x_j \, d x \geq x^*_j(\mu) - 2(n+1)\eps + \sum_{i = 1}^n x^*_i(\mu)\, .
\end{equation}
Passing to the limit $\delta, \eps \to 0^+$ in~\eqref{eq:upper2} and~\eqref{eq:lower2} shows that the limit of $\log_t \int_{\bm \Pcal(t)} e^{-\scalar{\bm \mu(t)^{-1} \bm c(t)}{x}} x_j \, d x$ exists and is equal to $x^*_j(\mu) + \sum_{i = 1}^n x^*_i(\mu)$.
\end{proof}

We remark that the quantity $\sum_{i = 1}^n x_i^*(\mu)$ which appears in Lemmas~\ref{lemma:volume} and~\ref{lemma:mean} corresponds to the \emph{tropical barycentric volume}, introduced by Loho and Schymura in~\cite{LohoSchymura20}, of the tropical polytope $\{ x \in \val \bm \Pcal \colon \tscalar{\val \bm c}{x} \leq \alpha \}$. We refer to~\cite{LohoSchymura20} for further properties on the tropical barycentric volume and its relation with the Euclidean volume.

\section{Tropical entropic central path with exponentially many linear pieces}\label{sec:cor}

Following Theorem~\ref{th:main}, the limit in~\eqref{eq:limit} only depends on $\mu = \val \bm \mu$, whence we call the function which maps $\mu$ to this limit the \emph{tropical entropic central path}. The same property holds when replacing the entropic barrier by the logarithmic one, which gives rise to a function that we call the \emph{tropical logarithmic central path} to avoid confusion. The latter has been studied in~\cite{ABGJ18} in a primal-dual setting with slack variables. This means that every point of the tropical logarithmic central path is a quadruple $(x,w,s,y)$ where $x$ and $y$ are respectively the primal and dual variables, and $w$ and $s$ are the associated slack variables.
\begin{corollary}\label{cor:same_path}
Let $\bm \Pcal \subset \K^n$ and $\bm c \in \K^n$ satisfying Assumption~\ref{assump}. Then, the tropical entropic central path is a piecewise linear function which coincides with the projection on the $x$-coordinates of the tropical logarithmic central path.
\end{corollary}

\begin{proof}
We introduce $\bm A \in \K^{m \times n}$ and $\bm b \in \K^m$ such that $\bm \Pcal = \{ \bm x \in \K_{\geq 0}^n \colon \bm A \bm x \leq \bm b\}$. We define $\overline{\bm \Pcal} \coloneqq \{ (\bm x, \bm w) \in \K_{\geq 0}^{n + m} \colon \bm A \bm x + \bm w = \bm b \}$ and $\overline{\bm \Qcal} \coloneqq \{ (\bm s, \bm y) \in \K_{\geq 0}^{n + m} \colon \bm s - \transpose{\bm A} \bm y = \bm c \}$ the primal and dual feasible sets (with slack variables). We apply the characterization of the tropical logarithmic central path given by~\cite[Theorem~15]{ABGJ18}. The latter requires that $\overline{\bm \Pcal}$ and $\overline{\bm \Qcal}$ contain points with only positive entries. Up to assuming that no row of $\bm A$ is equal to $0$, such a point of $\overline{\bm \Pcal}$ can be built by taking $\bm x \in \bm \Pcal$ such that $\val \bm x \in \interior \val \bm \Pcal$ (which is possible as $\val \bm \Pcal$ has nonempty interior by Assumption~\ref{assump}), and $\bm w \coloneqq \bm b - \bm A \bm x$. Indeed, the point $\bm x$ must satisfy $\bm A_i \bm x < \bm b_i$ and $\bm x_j > 0$ for all $(i,j) \in [m] \times [n]$, as a consequence of Proposition~\ref{prop:lift_interior}. Similarly, we exhibit a point with positive entries in $\overline{\bm \Qcal}$ by choosing $\bm y \coloneqq t^\eta$ where $\eta < \min_{ij} (\val \bm c_j - \val \bm A_{ij})$, so that $\val (\transpose{\bm A} \bm y)_j < \val(\bm c_j)$ for all $j \in [n]$, and the point $\bm s \coloneqq \bm c - \transpose{\bm A}\bm y$ has positive entries. 

Combining~\cite[Theorem~15]{ABGJ18} and~\cite[Proposition~14~(i)]{ABGJ18} shows that the primal part (\ie, the projection on the $(x,w)$ coordinates) of the point of tropical logarithmic central path with parameter $\mu$ is given by the tropical barycenter of the tropical polytope $\{(x,w) \in \val \overline{\bm \Pcal} \colon \tscalar{s^*}{x} \vee \tscalar{y^*}{w} \leq \mu \}$, where $(s^*, y^*)$ is the valuation of an (arbitrary) optimal solution of the following linear program over Puiseux series:
\[
\begin{array}{r@{\quad}l}
\text{minimize} & \scalar{\bm b}{\bm y} \\[\jot]
\text{subject to} & \bm s - \transpose{\bm A} \bm y = \bm c \, \\[\jot]
& (\bm s, \bm y) \in \K_{\geq 0}^{n + m} \, .
\end{array}
\]
As $\bm b \geq 0$ (recall that $0 \in \bm \Pcal$ by Assumption~\ref{assump}) and $\bm c \geq 0$, we can take $\bm s^* = \bm c$ and $\bm y^* = 0$, so that $\tscalar{s^*}{x} \vee \tscalar{y^*}{w} = \tscalar{\val \bm c}{x}$ for all $(x,w)$. Moreover, $\val \bm \Pcal$ is the projection on the $x$-coordinates of $\val \overline{\bm \Pcal}$. We deduce that the $x$-projection of the point of the  tropical logarithmic central path with parameter $\mu$ is the tropical barycenter of $\{ x \in \val \bm \Pcal \colon \tscalar{\val \bm c}{x} \leq \mu\}$, \ie~the point of the tropical entropic central path of parameter $\mu$ by Theorem~\ref{th:main}. The fact that the tropical entropic central path is a piecewise linear function then follows from~\cite[Proposition~16]{ABGJ18}.
\end{proof}

In~\cite{ABGJ18}, a pathological instance of linear program for which log-barrier interior point methods must perform an exponential number of iterations (w.r.t.~the dimension and the number of inequalities) is exhibited. The proof relies on the fact that the corresponding tropical logarithmic central path has exponentially many linear pieces. We exploit Corollary~\ref{cor:same_path} to extend this result to the tropical entropic central path. By comparison with~\cite{ABGJ18}, the objective function of the instance is modified to make sure that Assumption~\ref{assump}~\eqref{assump:iii} holds.

\begin{corollary}\label{cor:LW}
Consider the following parametric family of linear programs
\begin{equation}
\begin{array}{r@{\quad}l}
\text{minimize} & x_1 + t^{-1} x_2 + \sum_{j = 1}^{r-1} t^{-(j+1)} (x_{2j+1} + x_{2j+2}) \\[\jot]
\text{subject to} & 
x_1 \leq t^2 \, , \ \enspace x_2 \leq t, \\[\jot]
& x_{2j+1} \leq t \, x_{2j-1} \, , \enspace x_{2j+1} \leq t \, x_{2j} \tikzmark{} \\[\jot]
& x_{2j+2} \leq t^{1-1/2^j} (x_{2j-1} + x_{2j}) \tikzmark{} \\[\jot]
& x_{2r-1} \geq 0 \, , \enspace x_{2r} \geq 0 \, 
\end{array}
\insertbigbrace{$1 \leq j < r,$} \label{eq:LW}
\end{equation}
defined by $3r+1$ inequalities in dimension $2r$. The number of linear pieces in the subset of the tropical entropic central path with parameter $\mu \in [0,2]$ is equal to $2^{r-1}$.
\end{corollary}

\begin{proof}
Let $\bm \Pcal$ and $\bm c$ be the polyhedron and objective vector over Puiseux series associated with the parametric family of linear programs~\eqref{eq:LW}. Assumptions~\ref{assump}~\eqref{assump:i} and~\eqref{assump:iii} are trivially satisfied. As shown in~\cite[Section~4.3]{ABGJ18}, $\val \bm \Pcal$ is defined as the set of $x \in \T^n$ satisfying the following inequalities:
\begin{equation}
\begin{lgathered}
x_1 \leq 2 \, , \enspace x_2 \leq 1 \, , \\
x_{2j+1} \leq 1 + x_{2j-1} \, , \enspace  x_{2j+1} \leq 1 + x_{2j} \tikzmark{} \\
x_{2j+2} \leq (1-1/2^j) + \max(x_{2j-1}, x_{2j}) \tikzmark{}
\end{lgathered}
\insertbigbrace{$1 \leq j < r$,}
\end{equation}
and this set is regular, thus Assumption~\ref{assump}~\eqref{assump:ii} is satisfied. In~\cite [Proposition~20]{ABGJ18}, it is shown that the tropical barycenter of $\{x \in \val \bm \Pcal \colon x_1 \leq \mu\}$ is the point $\bar x(\mu)$ given by the following recursive equations:
\begin{align*}
\bar x_1(\mu) & = \min(\mu,2) \, , \\
\bar x_2(\mu) & = 1 \, , \\
\bar x_{2j+1}(\mu) & = 1 + \min(\bar x_{2j-1}(\mu), \bar x_{2j}(\mu)) \tikzmark{} \\
\bar x_{2j+2}(\mu) & = (1-1/2^j) + \max(\bar x_{2j-1}(\mu), \bar x_{2j}(\mu))  \tikzmark{} 
\insertbigbrace{$1 \leq j < r$.}
\end{align*}
The map $\mu \mapsto \bar x (\mu)$ is piecewise linear, and we  report from~\cite{ABGJ18} the values taken by the nondifferentiability points between $0$ and $2$ in Table~\ref{tab:values}. It is immediate to check that for all $\mu \in [0,2]$ and $i > 1$, $\val \bm c_i + \bar x_i(\mu) \leq \mu$. We deduce that for all such $\mu$, the tropical barycenter of $\{ x \in \val \bm \Pcal \colon x_1 \leq \mu\}$ belongs to $\{ x \in \val \bm \Pcal \colon \tscalar{\val \bm c}{x} \leq \mu\}$, and thus 
$x^*(\mu) = \bar x(\mu)$ as the latter set is included in the former set. The statement immediately follows from Table~\ref{tab:values}.
\end{proof}

\begin{table}
\caption{Nondifferentiability points and values of the maps $\mu \mapsto (\bar x_{2j+1}(\mu), \bar x_{2j+2}(\mu))$ where $1 \leq j < r$ and $k = 0, 1, \dots, 2^{j-1}-1$.}\label{tab:values}
\centering
$\renewcommand{\arraystretch}{1.75}%
\begin{array}{c@{\quad}c@{\quad}c@{\quad}c}
\toprule
\mu & \frac{2k}{2^{j-1}} & \frac{2k+1}{2^{j-1}} & \frac{2(k+1)}{2^{j-1}}  \\
\midrule
\bar x_{2j+1}(\mu) & j + \frac{2k}{2^j} & j + \frac{2k+2}{2^j} & j + \frac{2k+2}{2^j} \\
\bar x_{2j+2}(\mu) & j + \frac{2k+1}{2^j} & j + \frac{2k+1}{2^j} & j + \frac{2k+3}{2^j} \\
\bottomrule
\end{array}$
\end{table}

\bibliographystyle{alpha}
\bibliography{entropic_barrier}
\end{document}